\documentclass{amsart}
\usepackage{amssymb,amstext,amsmath,amscd,amsthm,amsfonts,enumerate,graphicx,latexsym,stmaryrd,multicol}

\usepackage[usenames]{color}
\usepackage[all]{xy}

\usepackage[top=34truemm,bottom=15truemm,left=25truemm,right=25truemm]{geometry}

\newtheorem{thm}{Theorem}[section]
\newtheorem{lem}[thm]{Lemma}
\newtheorem{prop}[thm]{Proposition}
\newtheorem{cor}[thm]{Corollary}
\theoremstyle{definition}
\newtheorem{dfn}[thm]{Definition}

\newtheorem{ex}[thm]{Example}

\newtheorem{rmk}[thm]{Remark}
\newtheorem*{claim}{Claim}
\theoremstyle{remark}

\newtheorem*{ac}{Acknowlegments}

\renewcommand{\qedsymbol}{$\blacksquare$}
\numberwithin{equation}{thm}

\def\Cok{\operatorname{Cok}}
\def\Cokv{\operatorname{\underline{Cok}}}
\def\depth{\operatorname{depth}}
\def\Ext{\operatorname{Ext}}

\def\G{\mathcal{G}}
\def\g{\mathrm{G}}
\def\ge{\geqslant}
\def\grade{\operatorname{grade}}
\def\height{\operatorname{ht}}
\def\Hom{\operatorname{Hom}}

\def\image{\operatorname{Im}}
\def\J{\mathrm{J}}

\def\Ker{\operatorname{Ker}}
\def\Kerv{\operatorname{\underline{Ker}}}
\def\le{\leqslant}
\def\lhom{\operatorname{\underline{Hom}}}
\def\lmod{\operatorname{\underline{mod}}}
\def\m{\mathfrak{m}}
\def\mod{\operatorname{mod}}

\def\p{\mathfrak{p}}

\def\proj{\operatorname{\mathsf{proj}}}

\def\s{\mathrm{S}}
\def\sgrade{\operatorname{s\text{.}grade}}

\def\syz{\Omega}
\def\t{\mathrm{T}}
\def\tf{\mathsf{TF}}

\def\tr{\operatorname{Tr}}

\def\X{\mathcal{X}}

\begin{document}
\allowdisplaybreaks
\title[Morphisms represented by monomorphisms with $n$-torsionfree cokernel]{Morphisms represented by monomorphisms with $n$-torsionfree cokernel}
\author{Yuya Otake}
\address{Graduate School of Mathematics, Nagoya University, Furocho, Chikusaku, Nagoya 464-8602, Japan}
\email{m21012v@math.nagoya-u.ac.jp}

\thanks{2020 {\em Mathematics Subject Classification.} 13D02, 16E05, 13C60, 16D90}
\thanks{{\em Key words and phrases.} morphism represented by monomorphisms, $n$-torsionfree module, stable category, syzygy, (Auslander) transpose, grade}
\begin{abstract}
We introduce and study a new class of morphisms which includes morphisms represented by monomorphisms in the sense of Auslander and Bridger.
As an application, we give not only an extension of Kato's theorem on morphisms represented by monomorphisms, but also a common generalization of several results due to Auslander and Bridger that describe relationships between torsionfreeness and the grades of Ext modules.
\end{abstract}
\maketitle
\section{Introduction}
Throughout this paper, let $R$ be a two-sided noetherian ring.
We assume that all modules are finitely generated right ones.
Denote by $\mod R$ the category of (finitely generated right) $R$-modules.
It is a natural and classical question to ask when a given homomorphism of $R$-modules is stably equivalent to another homomorphism satisfying certain good properties.
A well-studied one is about stable equivalence to a monomorphism: A homomorphism $f:X\to Y$ of $R$-modules is said to be {\em represented by monomorphisms} if there is an $R$-homomorphism $t:X\to P$ with $P$ projective such that $\binom{f}{t}:X\to Y\oplus P$ is a monomorphism.
This notion has been introduced by Auslander and Bridger \cite{AB}, and later studied in detail by Kato \cite{KK}.
Among other things, Kato gave the following characterization; we denote by $\tr(-)$ the (Auslander) transpose.

\begin{thm}[Kato]\label{Katointro}
Let $f:X\to Y$ be an $R$-homomorphism.
Then the following are equivalent.
\begin{enumerate}[\rm(1)]
    \item
    The morphism $f$ is represented by monomorphisms.
    \item
    The map $\Ext_{R^{\mathrm{op}}}^1(\tr f,R):\Ext_{R^{\mathrm{op}}}^1(\tr X,R)\to\Ext_{R^{\mathrm{op}}}^1(\tr Y,R)$ is injective.
\end{enumerate}
\end{thm}

Motivated by the above theorem, we define a condition which we call $(\t_n)$ for each integer $n\ge0$ so that $(\t_1)$ is equivalent to being represented by monomorphisms, and find out several properties.
The first main result is the following theorem, which recovers Theorem 1.1 as a special case.

\begin{thm}[Theorems \ref{main2} and \ref{main1}]\label{main12}
Let $n\ge1$ be an integer.
Consider the following conditions for an $R$-homomorphism $f:X\to Y$.
\begin{itemize}
    \item[$(a_n)$] The homomorphism $f$ satisfies $(\t_n)$.
    \item[$(a^\prime_n)$] There is an $R$-homomorphism $t:X\to P$ with $P$ projective such that the map $\binom{f}{t}:X\to Y\oplus P$ is a monomorphism whose cokernel is $(n-1)$-torsionfree and whose $R$-dual is surjective.
    \item[$(a^{\prime\prime}_n)$] There is an $R$-homomorphism $t:X\to P$ with $P$ projective such that the map $\binom{f}{t}:X\to Y\oplus P$ is a monomorphism whose cokernel is $(n-1)$-torsionfree.
    \item[$(b_n)$] The kernel of the map $(f, s):X\oplus Q\to Y$ is $n$-torsionfree for any epimorphism $s:Q\to Y$ with $Q$ projective.
    \item[$(c_n)$] The kernel of the map $\Ext_R^1(f,R):\Ext_R^1(Y,R)\to\Ext_R^1(X,R)$ has grade at least $n$.
\end{itemize}
Then the following implications hold.
$$
(a_n)\wedge(b_{n+1})\Longrightarrow(c_n),\qquad 
(b_n)\wedge(c_n)\Longrightarrow(a_n),\qquad
(a_n)\wedge(c_{n-1})\Longrightarrow(b_n).
$$
Moreover, the implications $(a_n)\Leftrightarrow(a^\prime_n)\Rightarrow(a^{\prime\prime}_n)$ hold, and the implication $(a^{\prime\prime}_n)\Rightarrow(a_n)$ is also true if the $(n-1)$-torsionfree property is closed under extensions.
\end{thm}

The notion of $n$-torsionfree modules was also introduced by Auslander and Bridger \cite{AB}, and played a central
role in the theory they developed.
For example, for an $R$-module $M$, Auslander and Bridger figured out the relationship between the grade of the $\Ext$ module $\Ext^i_R(M,R)$ and the torsionfreeness of the syzygy $\syz^i M$.
This result has been playing an important role in studies on $n$-torsionfree modules.
Theorem \ref{main12} can be regarded as a homomorphism version of Auslander and Bridger's theorem.
In fact, their theorem follows easily from Theorem \ref{main12}.
Moreover, we can state a higher version of Auslander and Bridger's theorem as follows.
It gives a common generalization of \cite[Propositions 2.26, 2.28 and Corollary 2.32]{AB}.

\begin{thm}[Theorem \ref{gradethm}]\label{gradethmintr}
Let $n\ge1$ and $j\ge0$ be integers and $M$ an $R$-module. 
The following are equivalent.
\begin{enumerate}[\rm(1)]
     \item
     The inequality $\grade_{R^{\mathrm{op}}}\Ext_R^i(M,R)\ge i+j-1$ holds for all $1\le i\le n$.
     \item
     The syzygy $\syz^i M$ is $i$-torsionfree and the natural map $\psi_M^i:\tr\syz^i\tr\syz^i M\to M$ satisfies $(\t_j)$ for all $1\le i\le n$.
\end{enumerate}
\end{thm}

Finally, let us consider the case where $R$ is commutative.
The local Gorensteinness of commutative rings is an actively studied subject in commutative algebra.
For example, the relationship between the local Gorensteinness of the ring $R$ and the structure of the category of $n$-torsionfree modules has been well-studied; see \cite{AB, DT, EG, GT, MTT} for instance.
Using the above theorem, we can also interpret the (local) Gorensteinness in terms of the natural maps $\psi_M^i:\tr\syz^i\tr\syz^i M\to M$ being represented by monomorphisms as follows.

\begin{cor}[Proposition \ref{psik} and Corollary \ref{S_nG_n-1}]\label{psiintr}
Suppose that $R$ is commutative.
\begin{enumerate}[\rm(1)]
     \item
     Assume that $R$ is local and with depth $t$.
     Let $k$ be the residue field of $R$.
     The following are equivalent.
     \begin{enumerate}[\rm(i)]
         \item
         The ring $R$ is Gorenstein.
         \item
        The natural map $\psi^{t+1}_M:\tr\syz^{t+1}\tr\syz^{t+1} M\to M$ is represented by monomorphisms for all $R$-modules $M$.
        \item
        The natural map $\psi_k^{t+1}:\tr\syz^{t+1}\tr\syz^{t+1} k\to k$ is represented by monomorphisms.
     \end{enumerate}
     \item
     Let $n\ge1$ be an integer.
     The following are equivalent.
     \begin{enumerate}[\rm(i)]
         \item
         The localization $R_\p$ is Gorenstein for all prime ideals $\p$ with $\depth R_{\p}<n$
         \item
         The natural map $\psi^{i}_M:\tr\syz^i\tr\syz^i M\to M$ is represented by monomorphisms for all $R$-modules $M$ and for all integers $1\le i\le n$.
     \end{enumerate}
\end{enumerate}
\end{cor}

The organization of this paper is as follows.
In Section 2, we state several notions and their basic properties for later use.
In Section 3, we give the definition of the condition $(\t_n)$.
We derive its basic properties and prove Theorem \ref{main12}.
In Section 4, we recall the construction of the natural map $\psi^i_{M}:\tr\syz^i\tr\syz^i M\to M$, and apply a result obtained in Section 3 to $\psi^i_M$.
We show Theorem \ref{gradethmintr} in this section.
In the final Section 5, we describe local Gorensteinness of commutative rings in terms of $\psi^i_M$ being represented by monomorphisms.

\section{Preliminaries}
In this section, we introduce some notions and terminologies.
For an $R$-homomorphism $f$, we denote by $\Cok f$ the cokernel of $f$.
Also, we denote by $(-)^\ast$ the $R$-dual functor $\Hom_R(-,R)$.
We begin with fixing some notation.

\begin{dfn}\label{dfn1.1}
\begin{enumerate}[\rm(1)]
    \item
    We denote by $\lmod R$ the {\em stable category} of $\mod R$.
    The objects of $\lmod R$ are the same as those of $\mod R$.
    The morphism set of objects $X, Y$ of $\lmod R$ is defined by
    $$
    \Hom_{\lmod R}(X,Y)
    =\lhom_R(X,Y)
    =\Hom_R(X,Y)/{\mathcal{P}(X,Y)},
    $$
    where $\mathcal{P}(X,Y)$ is the set of $R$-homomorphisms $X\to Y$ factoring through some projective modules.
    For any homomorphism $f:X\to Y$, we denote by $\underline{f}$ the image of $f$ in $\lhom_R(X,Y)$.
    \item
    Let $X$ be an $R$-module and $\cdots\to P_n\xrightarrow{\partial_n}P_{n-1}\to\cdots\to P_1\xrightarrow{\partial_1}P_0\xrightarrow{\partial_0}X\to0$ a projective resolution of $X$.
    \begin{itemize}
        \item[(2-a)]
        Let $n\ge0$ be an integer.
        The {\em $n$th syzygy} $\syz^n X$ of $X$ is defined as $\image \partial_n$.
        Then $\syz^n X$ is uniquely determined by $X$ up to projective summands.
        Note that taking the $n$th syzygy induces an additive functor $\syz^n:\lmod R\to\lmod R$.
        \item[(2-b)]
        The {\em transpose} $\tr X$ of $X$ is defined as $\Cok \partial_1^\ast$.
        Then $\tr X$ is uniquely determined by $X$ up to projective summands.
        Note that taking the transpose induces an additive functor $\tr:\lmod R\to\lmod R^{\mathrm{op}}$.
    \end{itemize}
\end{enumerate}
\end{dfn}

In general, the notion of a left (resp. right) $\X$-approximation is defined for any full subcategory $\X$ of $\mod R$.
For the details, we refer to \cite{AS}.
Denote by $\proj R$ the full subcategory of $\mod R$ consisting of projective modules.
We can state the definition of a left (resp. right) $\proj R$-approximation as follows.

\begin{dfn}\label{dfn1.2}
 A homomorphism $f: P\to X$ (resp. $f:X\to P$) of $R$-modules is said to be a {\em right} (resp. {\em left}) {\em $\proj R$-approximation} of $X$ if $P\in\proj R$ and $f$ (resp. $f^\ast$) is surjective.
\end{dfn}

We frequently use the following terminologies about morphisms in $\lmod R$.

\begin{dfn}\label{dfn1.3}
\begin{enumerate}[\rm(1)]
    \item
    A homomorphism $f:X\to Y$ of $R$-modules is called a {\em stable isomorphism} if $\underline{f}$ is an isomorphism in the category $\lmod R$, that is, if there exists a homomorphism $g:Y\to X$ such that $\underline{gf}=\underline{1_X}$ and $\underline{fg}=\underline{1_Y}$.
    When this is the case, we say that $X$ and $Y$ are {\em stably isomorphic}, and write $X\approx Y$.
    \item
    Let $f:X\to Y$ and $g:X^\prime\to Y^\prime$ be $R$-homomorphisms.
    By $f\approx g$ we mean that there are stable isomorphisms $s:X\to X^\prime$ and $t:Y\to Y^\prime$ such that $\underline{tf}=\underline{gs}$.
    \item\cite[Definition and Lemma 2.11]{KK}
    Let $f:X\to Y$ be a homomorphism of $R$-modules.
    Let $s:P\to Y$ (resp. $s:X\to P$) be a right (resp. left) $\proj R$-approximation of $Y$ (resp. $X$).
    The module $\Kerv f$ (resp. $\Cokv f$) is defined as $\Ker\left((f,s):X\oplus P\to Y\right)$ (resp. $\Cok\left(\scriptsize{\begin{pmatrix}f \\s\end{pmatrix}}:X\to Y\oplus P\right)$).
    The module $\Kerv f$ (resp. $\Cokv f$) is uniquely determined by $f$ up to projective summands.
\end{enumerate}
\end{dfn}

Here is a collection of some statements on the notions introduced above.

\begin{rmk}\label{rmkst}
\begin{enumerate}[\rm(1)]
    \item
    Let $X$ and $Y$ be $R$-modules.
    Then $X\approx Y$ if and only if $X\oplus P\cong Y\oplus Q$ for some projective $R$-modules $P, Q$.
    \item
    Let $f:X\to Y$ and $g:X^\prime\to Y^\prime$ be $R$-homomorphisms.
    If $f\approx g$, then $\Kerv f\approx\Kerv g$ and $\Cokv f\approx\Cokv g$.
    If $f$ is surjective, then $\Kerv f\approx\Ker f$ .
    See \cite[Definition and Lemma 2.11]{KK} for details.
    \item
    Let $0\to X\xrightarrow{f}Y\xrightarrow{g}Z\to0$ be an exact sequence in $\mod R$.
    Then there exists an exact sequence $0\to\syz Z\to X\oplus P\xrightarrow{(f,t)}Y\to0$ in $\mod R$ with $P$ projective.
    Doing the same, we get an exact sequence $0\to\syz Y\xrightarrow{h}\syz Z\oplus Q\to X\oplus P\to0$, where $Q$ is projective and $h\approx\syz g$.
    It naturally gives an exact sequence $0\to\syz Y\to\syz Z\oplus F\to X\to0$ with $F$ projective.
\end{enumerate}
\end{rmk}

The notion of $n$-torsionfree modules was introduced by Auslander and Bridger \cite{AB}, and played a central role in the stable module theory they developed.
It is also an important concept in this paper.

\begin{dfn}\label{dfn1.4}
\begin{enumerate}[\rm(1)]
    \item
    Let $n\ge0$ be an integer.
    An $R$-module $X$ is said to be {\em $n$-torsionfree} if $\Ext_{R^{\mathrm{op}}}^i(\tr X,R)=0$ for all $1\le i\le n$.
    We denote by $\tf_n(R)$ the full subcategory of $\mod R$ consisting of $n$-torsionfree modules.
    \item
    A full subcategory $\X$ of $\mod R$ is said to be {\em closed under extensions} if for every exact sequence $0\to X\to Y\to Z\to0$ in $\mod R$ with $X, Z\in\X$, it holds that $Y\in\X$.
\end{enumerate}
\end{dfn}

We recall the definitions of heights and depths.
The details can be found in \cite{Mat}.

\begin{dfn}
Suppose that $R$ is commutative.
\begin{enumerate}[\rm(1)]
    \item
    The {\em height} of a prime ideal $\p$ of $R$, denoted by $\height\p$, is defined to be the supremum of the lengths of all strictly decreasing chains of prime ideals contained in $\p$.
    \item
    When $R$ is a local ring with maximal ideal $\m$, the {\em depth} of an $R$-module $M$ is defined to be the infimum of integers $i$ such that $\Ext^i_R(R/{\m}, M)\ne0$, and denoted by $\depth_R M$.
\end{enumerate}
\end{dfn}

We recall the definitions of Serre's condition $(\s_n)$ and the local Gorensteinness condition $(\g_n)$.
These conditions are related to the $n$-torsionfreeness of modules; see \cite{AB, EG, GT, MTT} for instance.

\begin{dfn}
Suppose that $R$ is commutative.
Let $n$ be an integer.
\begin{enumerate}[\rm(1)]
    \item
    Let $M$ be an $R$-module.
    We say that $M$ satisfies {\em Serre's condition} $(\s_n)$ if the inequality $\depth_{R_\p}M_\p\ge\min\{n,\height\p\}$ holds for all prime ideals $\p$ of $R$.
    \item
    We say that $R$ satisfies $(\g_n)$ if the localization $R_\p$ is a Gorenstein local ring for all prime ideals $\p$ of $R$ with $\height\p\le n$.
\end{enumerate}
\end{dfn}

\begin{rmk}\label{tfrmk}
Let $n\ge0$ be an integer.
\begin{enumerate}[\rm(1)]
    \item
    Let $M$ be an $R$-module.
    We denote by $\varphi_M:M\to M^{\ast\ast}$ the canonical map given by $\varphi_M(x)(f)=f(x)$ for $x\in M$ and $f\in M^\ast$.
    By \cite[Proposition 2.6]{AB}, $M$ is $1$-torsionfree if and only if $\varphi_M$ is injective, that is, $M$ is torsionless.
    Similarly, $M$ is $2$-torsionfree if and only if $\varphi_M$ is bijective, that is, $M$ is reflexive.
    \item
    All $n$-torsionfree modules are $n$-syzygy; see \cite[Theorem 2.17]{AB}.
    \item
    Suppose that $R$ is commutative.
    If $R$ satisfies $(\s_n)$, then so do all the $n$-syzygy modules.
    This assertion is proved by the depth lemma \cite[Proposition 1.2.9]{BH}.
\end{enumerate}
\end{rmk}

\section{The condition $(\t_n)$}
We begin with introducing the new condition $(\t_n)$ for $R$-homomorphisms.
It is a natural extension of the condition (2) in Theorem \ref{Katointro}.

\begin{dfn}\label{Tn}
Let $n\ge1$ be an integer.
We say that a homomorphism $f:X\to Y$ of $R$-modules satisfies $(\t_n)$ if the map $\Ext^i_{R^{\mathrm{op}}}(\tr f,R)$ is bijective for all $1\le i\le n-1$ and $\Ext^n_{R^{\mathrm{op}}}(\tr f,R)$ is injective.
In addition, we provide that every $R$-homomorphism satisfies $(\t_0)$ and an $R$-homomorphism $f$ satisfies $(\t_{\infty})$ if $f$ satisfies $(\t_n)$ for all $n\ge0$. 
\end{dfn}

In this section, we explore those homomorphisms which satisfy the condition $(\t_n)$.
Below we give several statements which easily follow by definition.

\begin{rmk}\label{bydef}
Let $f:X\to Y$, $f^\prime:X^\prime\to Y^\prime$ and $g:Y\to Z$ be $R$-homomorphisms. Let $n\ge1$ be an integer.
\begin{enumerate}
    \item
    Suppose that $f\approx f^\prime$. Then $f$ satisfies $(\t_n)$ if and only if so does $f^\prime$.
    \item 
    If $g$ and $f$ satisfy $(\t_n)$, then $gf$ satisfies $(\t_n)$.
    \item
    If $gf$ satisfies $(\t_1)$, then $f$ satisfies $(\t_1)$.
    In general, if $gf$ and $g$ satisfy $(\t_n)$, then $f$ satisfies $(\t_n)$.
    \item
    Consider the following conditions.
    \begin{itemize}
        \item[$(a_n)$] The homomorphism $f$ satisfies $(\t_n)$.
        \item[$(b_n)$] The module $X$ is $n$-torsionfree.
        \item[$(c_n)$] The module $Y$ is $n$-torsionfree.
    \end{itemize}
    Then the following implications hold.
    $$
    (a_{n+1})\wedge(b_n)\Longrightarrow(c_n),\qquad 
    (b_n)\wedge(c_{n-1})\Longrightarrow(a_n),\qquad
    (a_n)\wedge(c_n)\Longrightarrow(b_n).
    $$
    \item
    Stable isomorphisms satisfy $(\t_\infty)$.
\end{enumerate}
\end{rmk}

Now, let us investigate the condition $(\t_1)$.
Kato \cite[Theorem 3.9]{KK} proved the following theorem, which asserts that $(\t_1)$ is equivalent to being represented by monomorphisms, using the homotopy theory developed in \cite[Sections 2 and 3]{KK}.
Using \cite[Proposition 2.6]{AB}, we can provide another simpler proof of Kato's theorem as follows.

\begin{thm}[Kato]\label{T_1}
Let $f:X\to Y$ be a homomorphism of $R$-modules.
The following are equivalent.
\begin{enumerate}[\rm(1)]
    \item
    The homomorphism $f$ is represented by monomorphisms, i.e., there is an $R$-homomorphism $t:X\to P$ with $P$ projective such that $\binom{f}{t}:X\to Y\oplus P$ is a monomorphism.
    \item
    The homomorphism $f$ satisfies $(\t_1)$.
    \item
    One has $\Ker f\cap\Ker\varphi_X=0$.
\end{enumerate}
\end{thm}

\begin{proof}
We have a commutative diagram
\begin{equation*}
\xymatrix@R-1pc@C-1pc{
&0\ar[d]&&0\ar[d]&& \\
0\ar[r]&\Ker f\cap\Ker\varphi_X\ar[rr]\ar[dd]&&\Ker f\ar[dd]&& \\
&&&&& \\
0\ar[r]&\Ker\varphi_X\ar[rr]\ar[dd]^{\hat{f}}&&X\ar[rr]^{\varphi_X}\ar[dd]^{f}&&X^{\ast\ast}\ar[dd]  \\
&&&&& \\
0\ar[r]&\Ker\varphi_Y\ar[rr]&&Y\ar[rr]^{\varphi_Y}&&Y^{\ast\ast}
}
\end{equation*}
with exact rows and columns.
By \cite[Proposition 2.6(a)]{AB}, there is a commutative diagram:
\begin{equation*}
\xymatrix@R-1pc@C-1pc{
\Ker\varphi_X\ar[rrr]^-{\sim}\ar[dd]^{\hat{f}}&&&\Ext^1(\tr X,R)\ar[dd]^{\Ext^1(\tr f,R)} \\
&& \\
\Ker\varphi_Y\ar[rrr]^-{\sim}&&&\Ext^1(\tr Y,R).
}
\end{equation*}
Therefore, one has $\Ker f\cap\Ker\varphi_X\cong\Ker\Ext^1(\tr f,R)$ and we have the equivalence $(2)\Leftrightarrow(3)$.

Next, suppose that $(1)$ holds, i.e., there is a homomorphism $t:X\to P$ with $P$ projective such that $\binom{f}{t}$ is injective.
As $\Ker \binom{f}{t}\cap\Ker\varphi_X=0$, we obtain that $\binom{f}{t}$ satisfies $(\t_1)$.
Thus $f$ also satisfies $(\t_1)$, and the implication $(1)\Rightarrow(2)$ follows.

Finally, assume that $(3)$ holds and take a left $\proj R$-approximation $t:X\to P$ of $X$.
Then we see that $\Ker\varphi_X=\Ker t$.
We have $\Ker \binom{f}{t}=\Ker f\cap\Ker t=\Ker f\cap\Ker\varphi_X=0$, that is, $\binom{f}{t}$ is injective.
Therefore, the implication $(3)\Rightarrow(1)$ holds and we are done.
\end{proof}

If a homomorphism $f$ satisfies $(\t_1)$, then there is a close relationship between $\Kerv f$ and $\Cokv f$  as follows.

\begin{lem}\cite[Lemma 2.17, Theorem 4.12]{KK}\label{diag}
Let $f:X\to Y$ be a homomorphism of $R$-modules.
Let $s:X\to P$ be a left $\proj R$-approximation of $X$ and $t:Q\to Y$ a right $\proj R$-approximation of $Y$.
Then there exist exact sequences
\begin{equation}\label{diag1.1}
0\to\Ker f\to\Kerv f\to Q\to\Cok f\to0
\end{equation}
\begin{equation}\label{diag1.2}
0\to\Ker f/(\Ker f\cap\Ker s)\to P\to\Cokv f\to\Cok f\to0.
\end{equation}
In particular, if $f$ satisfies $(\t_1)$, then $0\to\Ker f\to P\to\Cokv f\to\Cok f\to0$ is exact and $\syz\Cokv f\approx\Kerv f$ .
\end{lem}

\begin{proof}
For the convenience of the reader, we give a proof of the lemma.
There are two commutative diagrams
$$
\xymatrix@R-1pc@C-1pc{
&&&0\ar[d]&&0\ar[d]&&&0\ar[d]&&0\ar[d]&&& \\
0\ar[r]&\Ker f\ar[rr]\ar[dd]&&X\ar[rr]^-{\tilde{f}}\ar[dd]^-{\scriptsize \begin{pmatrix} 1 \\ 0\end{pmatrix}}&&\image f\ar[r]\ar[dd]&0 &&\Ker f/W\ar[dd]&&P\ar[dd]^-{\scriptsize \begin{pmatrix} 0 \\ 1\end{pmatrix}}&&&\\
&&&&
&&&&&&\\
0\ar[r]&\Kerv f\ar[rr]&&X\oplus Q\ar[rr]^-{(f, t)}\ar[dd]^-{(0, 1)}&&Y\ar[r]\ar[dd]&0
&0\ar[r]&\image \scriptsize{\begin{pmatrix} f \\ s \end{pmatrix}}\ar[rr]\ar[dd]^-{p}&&Y\oplus P\ar[rr]\ar[dd]^{(1,0)}&&\Cok \scriptsize{\begin{pmatrix} f \\ s \end{pmatrix}}\ar[r]\ar[dd]&0\\
&&&&
&&&&&&\\
&&&Q\ar[rr]\ar[d]&&\Cok f\ar[r]\ar[d]&0 
&0\ar[r]&\image f\ar[rr]\ar[d]&&Y\ar[rr]\ar[d]&&\Cok f\ar[r]\ar[d]&0\\
&&&0&&0& 
&&0&&0&&0&\\
}
$$
with exact rows and columns, where $W=\Ker f\cap\Ker s$ and $p:\image \binom{f}{s} \to \image f$ is the homomorphism defined by $p\scriptsize{\begin{pmatrix}  f(x) \\ s(x)\end{pmatrix}}=f(x)$.
Applying the snake lemma to each, we get the desired exact sequences (\ref{diag1.1}) and (\ref{diag1.2}).

Suppose that $f$ satisfies $(\t_1)$.
Then
\begin{equation}\label{diag1.3}
0\to\Ker f\to P\to\Cokv f\to\Cok f\to0
\end{equation}
is exact by (\ref{diag1.2}) and Theorem \ref{T_1}.
Now let $Q\xrightarrow{t} Y$ be a right $\proj R$-approximation of $Y$.
Applying (\ref{diag1.3}) to $(f, t)$, we obtain an exact sequence $0\to\Ker(f,t)\to P^\prime\to\Cokv (f,t)\to0$ with $P^\prime$ projective.
We get that $\Kerv f\approx\Ker(f,t)\approx\syz\Cokv(f,t)\approx\syz\Cokv f$.
\end{proof}

For an injective or surjective homomorphism $f:X\to Y$ in $\mod R$, the following relationship holds between $f$ being a stable isomorphism and $\tr f$ satisfying $(\t_1)$. 

\begin{lem}\label{stisolem}
Let $M$ be an $R$-module.
\begin{enumerate}[\rm(1)]
    \item
    Let $0\to K\to X\xrightarrow{f}Y\to0$ be an exact sequence.
    Then $f$ is a stable isomorphism if and only if $K$ is projective and $\tr f$ satisfies $(\t_1)$.
    Moreover, when this is the case, $f$ is a split epimorphism.
    \item
    Let $0\to X\xrightarrow{f}Y\to C\to0$ be an exact sequence.
    Then $f$ is a stable isomorphism if and only if $C$ has projective dimension at most one and $\tr f$ satisfies $(\t_1)$.
\end{enumerate}
\end{lem}

\begin{proof}
(1) Suppose that $f$ is a stable isomorphism.
Then $\tr f$ is also a stable isomorphism, and so satisfies $(\t_1)$.
Furthermore, the exact sequence $0\to K\to X\xrightarrow{f}Y\to0$ induces an exact sequence of functors
\begin{equation}\label{stisolem1.1}
\xymatrix@R-1pc@C-1pc{
\Ext^1(Y,-)\ar[rr]^-{\Ext^1(f,-)}&&\Ext^1(X,-)\ar[r]&\Ext^1(K,-)\ar[r]&\Ext^2(Y,-)\ar[rr]^-{\Ext^2(f,-)}&&\Ext^2(X,-)
}.
\end{equation}
Since the natural transformation $\Ext^i(f,-)$ is an isomorphism for all $i>0$, one has $\Ext^1(K,-)=0$, that is, $K$ is projective.
Also, since $0\to\Hom(Y,K)\to\Hom(X,K)\to\Hom(K,K)\to0$ is exact by (\ref{stisolem1.1}), $f$ is a split epimorphism.
Conversely, assume that $K$ is projective and $\tr f$ satisfies $(\t_1)$.
The latter means that $\Ext^1(f,R)$ is injective.
We obtain an exact sequence $0\to\tr Y\xrightarrow{t_f}T_X\to\tr K\to0$, where $T_X\approx\tr X$ and $t_f\approx\tr f$ by \cite[Lemma 3.9]{AB}.
As $\tr K\approx0$, $t_f$ is a stable isomorphism, and so is $f$.

(2) By Remark \ref{rmkst}(3), there exists an exact sequence $0\to\syz C\to X\oplus P\xrightarrow{(f, t)}Y\to0$ with $P$ projective. Thus the assertion holds by (1).
\end{proof}

Let us consider the condition $(\t_n)$ in the case where $n$ is arbitrary.
Let $M$ be an $R$-module.
The {\em grade} of $M$, which is denoted by $\grade_R M$, is defined to be the infimum of integers $i$ such that $\Ext_R^i(M,R)=0$.
The relationship between the grades of Ext modules and the torsionfreeness of modules has been actively studied; the works of Auslander and Bridger \cite{AB} and Auslander and Reiten \cite{AR} are among the most celebrated studies.
The following theorem is one of the main results of this paper, which interprets the condition $(\t_n)$ in terms of grades and torsionfreeness.
In Section 4, using the following theorem, we recover the results of Auslander and Bridger \cite{AB} and Auslander and Reiten \cite{AR}.

\begin{thm}\label{main2}
Let $n\ge1$ be an integer. Consider the following conditions for a homomorphism $f:X\to Y$ of $R$-modules.
\begin{itemize}
    \item[$(a_n)$] The homomorphism $f$ satisfies $(\t_n)$.
    \item[$(b_n)$] The $R$-module $\Kerv f$ is $n$-torsionfree.
    \item[$(c_n)$] There is an inequality $\grade_{R^{\mathrm{op}}}\Ker\Ext_R^1(f,R)\ge n$.
\end{itemize}
Then the following implications hold.
$$
(a_n)\wedge(b_{n+1})\Longrightarrow(c_n),\qquad 
(b_n)\wedge(c_n)\Longrightarrow(a_n),\qquad
(a_n)\wedge(c_{n-1})\Longrightarrow(b_n).
$$
\end{thm}

\begin{proof}
Let $s:P\to Y$ be an epimorphism with $P$ projective. By replacing $f$ with $(f, s)$, we may assume that $f$ is surjective, and that $\Ker f\approx\Kerv f$.
Then, by \cite[Lemma 3.9]{AB}, there exists a commutative diagram with exact rows
$$
\xymatrix@R-1pc@C-1pc{
0\ar[rr]&&Y^\ast\ar[rr]\ar@{=}[dd]&&X^\ast\ar[rr]\ar@{=}[dd]&&K^\ast\ar[rr]\ar@{=}[dd]&&\Ext^1(Y,R)\ar[rr]^-{\Ext^1(f,R)}&&\Ext^1(X,R)  \\
&&&&&&&&&&& \\
0\ar[rr]&&Y^\ast\ar[rr]&&X^\ast\ar[rr]&&K^\ast\ar[rr]^{a}&&\tr Y\ar[rr]^{t_f}&&T_X\ar[rr]&&\tr K\ar[rr]&&0,
}
$$
where $K=\Ker f$, such that $T_X\approx\tr X$ and $t_f\approx\tr f$.
Moreover, the sequence $0\to(\tr K)^\ast\to(T_X)^\ast\xrightarrow{(t_f)^\ast}(\tr Y)^\ast\to0$ is also exact. 
Let $U=\Ker\Ext^1(f,R)$ and $V=\Cok a$.
We decompose the exact sequence $K^\ast\xrightarrow{a}\tr Y\xrightarrow{t_f}T_X\to\tr K\to0$ into short exact sequences
$$
0\to U\to\tr Y\xrightarrow{b}V\to0,\qquad 0\to V\xrightarrow{c}T_X\to\tr K\to0.
$$
As the homomorphisms $b^\ast$ and $c^\ast$ are surjective, we get two long exact sequences
\begin{equation}\label{1.1}
\xymatrix@R-1pc@C-1pc{
    &0\ar[r]&U^\ast\ar[r]&\Ext^1(V,R)\ar[rr]^-{\Ext^1(b,R)}&&\Ext^1(\tr Y,R)\ar[rr]&&\Ext^1(U,R)\ar[rr]&&\cdots\\
}
\end{equation}
and
\begin{equation}\label{1.2}
\xymatrix@R-1pc@C-1pc{
    0\ar[rr]&&\Ext^1(\tr K,R)\ar[rr]&&\Ext^1(T_X,R)\ar[rr]^-{\Ext^1(c,R)}&&\Ext^1(V,R)\ar[rr]&&\cdots\\
}.
\end{equation}
The long exact sequence (\ref{1.1}) indicates that the condition $(c_n)$ holds if and only if $\Ext^i(b,R)$ is bijective for all $1\le i\le n-1$ and $\Ext^n(b,R)$ is injective, while (\ref{1.2}) shows that the condition $(b_n)$ holds if and only if $\Ext^i(c,R)$ is bijective for all $1\le i\le n-1$ and $\Ext^n(c,R)$ is injective.
Note that we have the following commutative diagram for any integer $i$.
\begin{equation}\label{1.3}
\xymatrix@R-1pc@C-1pc{
\Ext^i(T_X,R)\ar[rrrr]^-{\Ext^i(t_f,R)}\ar[rrd]_-{\Ext^i(c,R)}&&&&\Ext^i(\tr Y,R)\\
&&\Ext^i(V,R)\ar[rru]_-{\Ext^i(b,R)}&&&
}
\end{equation}
In particular, $\Ext^i(c,R)$ is injective for all $1\le i\le n$ and $\Ext^j(b,R)$ is surjective for all $1\le j\le n-1$ if the condition $(a_n)$ holds.

Suppose that the condition $(b_n)\wedge(c_n)$ holds.
Then $\Ext^i(c,R)$ and $\Ext^i(b,R)$ are bijective for all $1\le i\le n-1$, moreover, $\Ext^n(b,R)$ and $\Ext^n(c,R)$ are injective.
Thus $\Ext^i(\tr f,R)$ is bijective for all $1\le i\le n-1$ and $\Ext^n(\tr f,R)$ is injective by (\ref{1.3}), that is, the condition $(a_n)$ holds.

Next, we assume that the condition $(a_n)\wedge(c_{n-1})$ holds.
Then $\Ext^i(b,R)$ is bijective for all $1\le i\le n-1$, and $\Ext^j(c,R)$ is injective for all $1\le j\le n$.
For all $1\le i\le n-1$, the map $\Ext^i(c,R)$ is bijective by (\ref{1.3}) since $\Ext^i(\tr f,R)$ is bijective.
As $\Ext^i(c,R)$ is bijective for all $1\le i\le n-1$ and $\Ext^n(c,R)$ is injective, the condition $(b_n)$ holds.

Finally, we prove that $(a_n)\wedge(b_{n+1})$ implies $(c_n)$. The condition $(a_n)$ deduces that $\Ext^i(b,R)$ is surjective for all $1\le i\le n-1$, and $(b_{n+1})$ yields that $\Ext^j(c,R)$ is bijective for all $1\le j\le n$. 
Thus, if the condition $(a_n)\wedge(b_{n+1})$ holds, then $\Ext^i(b,R)$ is bijective for all $1\le i\le n-1$ and $\Ext^n(b,R)$ is injective since $\Ext^j(\tr f,R)$ is bijective for all $1\le j\le n-1$ and $\Ext^n(\tr f,R)$ is injective. 
Therefore, the condition $(c_n)$ holds, and the proof of the theorem is completed.
\end{proof}

\begin{rmk}\label{surj}
For Theorem \ref{main2}, we also consider the condition $$(b^\prime_n)\mbox{ : }\Ker f\mbox{ is } n\mbox{-torsionfree}.$$
We note the following.
\begin{enumerate}[\rm(1)]
    \item
    From (\ref{diag1.1}), the implication $(b_1)\Rightarrow(b^\prime_1)$ holds.
    If $\tf_1(R)$ is closed under extensions, then the opposite implication $(b^\prime_1)\Rightarrow(b_1)$ holds.
    \item
    Let $n\ge1$ be an integer.
    Suppose that $f$ is surjective.
    \begin{itemize}
        \item[(2-a)]
        By Remark \ref{rmkst}(2), the conditions $(b_n)$ and $(b^\prime_n)$ are equivalent.
        \item[(2-b)]
        Let $\iota:\Ker f\to X$ be the inclusion map.
        If the $R$-dual map $\iota^\ast$ is surjective, then the conditions $(a_n)$ and $(b^\prime_n)$ are equivalent.
    \end{itemize}
\end{enumerate}
\end{rmk}

The following corollary is none other than \cite[Theorem 4.2]{KK}, which gives a simple characterization of the morphisms represented by monomorphisms when $R$ is commutative and generically Gorenstein (e.g., when $R$ is an integral domain).
We can deduce it from Theorem \ref{main2} and \cite[Proposition 4.21]{AB}.

\begin{cor}[Kato]\label{KKmain}
Suppose that $R$ is commutative and satisfies $(\s_1)$ and $(\g_0)$.
Let $f:X\to Y$ be a homomorphism of $R$-modules.
Then $f$ satisfies $(\t_1)$ if and only if $\Ker f$ is torsionless, if and only if $\Kerv f$ is torsionless.
\end{cor}

\begin{proof}
Since $R$ satisfies $(\s_1)$ and $(\g_0)$, $\tf_1(R)$ is closed under extensions; see \cite[Theorem 2.3]{MTT} for instance.
Hence, by Remark \ref{surj}(1), $\Ker f$ is torsionless if and only if so is $\Kerv f$.
Suppose that $\Kerv f$ is torsionless.
By \cite[Proposition 4.21]{AB}, we have $\grade\Ker\Ext^1(f,R)\ge\grade\Ext^1(Y,R)\ge1$.
It follows from Theorem \ref{main2} that $f$ satisfies $(\t_1)$.
\end{proof}

The {\em strong grade} of an $R$-module $M$, denoted by $\sgrade_R M$, is defined to be the supremum of integers $i$ such that $\grade_R N\ge i$ for all submodules $N$ of $M$.
The following theorem gives a detailed characterization of those homomorphisms which satisfy the condition $(\t_n)$ and can also be regarded as a generalization of Theorem \ref{T_1}. 

\begin{thm}\label{main1}
Let $n\ge1$ be an integer. Consider the following conditions for a homomorphism $f:X\to Y$ of $R$-modules.
\begin{enumerate}[\rm(1)]
\item
The homomorphism $f$ satisfies $(\t_n)$.
\item
The homomorphism $f$ satisfies $(\t_1)$ and $\Cokv f$ is $(n-1)$-torsionfree.
\item
There exists an exact sequence $0\to X\xrightarrow{\binom{f}{t}}Y\oplus P\to Z\to0$
such that $P$ is projective, $Z$ is $(n-1)$-torsionfree and $\binom{f}{t}^\ast$ is surjective.
\item
There exists an exact sequence $0\to X\xrightarrow{\binom{f}{t}}Y\oplus P\to Z\to0$
such that $P$ is projective and $Z$ is $(n-1)$-torsionfree.
\item
There exists an exact sequence $0\to X^\prime\xrightarrow{f^\prime}Y^\prime\to Z^\prime\to0$
such that $f\approx f^\prime$, $Z^\prime$ is $(n-1)$-torsionfree and $f^{\prime\ast}$ is surjective.
\item
There exists an exact sequence $0\to X^\prime\xrightarrow{f^\prime}Y^\prime\to Z^\prime\to0$
such that $f\approx f^\prime$ and $Z^\prime$ is $(n-1)$-torsionfree.
\end{enumerate}
Then the implications 
$$(1)\Longleftrightarrow(2)\Longleftrightarrow(3)\Longleftrightarrow(5)\Longrightarrow(4)\Longrightarrow(6)$$
hold.
The implication $(6)\Rightarrow(1)$ holds as well, if one requires the module $Z^\prime$ in $(6)$ to satisfy the inequality $\sgrade_{R^{\mathrm{op}}} \Ext_R^1(Z^\prime,R)\ge n-1$. 
In particular, the implication $(6)\Rightarrow(1)$ holds if $\tf_{n-1}(R)$ is closed under extensions.
Conversely, if the implication $(6)\Rightarrow(1)$ holds for every homomorphism $f:X\to Y$ with $X$ $(n-1)$-torsionfree, then $\tf_{n-1}(R)$ is closed under extensions.
\end{thm}

\begin{proof}
We first prove the following claim.

\begin{claim}
Let $0\to X^\prime\xrightarrow{f^\prime}Y^\prime\to Z^\prime\to0$ be an exact sequence with $f^{\prime\ast}$ surjective.
Then $f^\prime$ satisfies $(\t_n)$ if and only if $Z^\prime$ is $(n-1)$-torsionfree.
\end{claim}

\begin{proof}[Proof of Claim]
By \cite[Lemma 3.9]{AB}, we get an exact sequence 
$0\to\tr Z^\prime\to T_{Y^\prime}\xrightarrow{t_{f^\prime}}\tr X^\prime\to0$, where $T_{Y^\prime}\approx\tr Y^\prime$ and $t_{f^\prime}\approx\tr f^\prime$, and whose $R$-dual $0\to(\tr X^\prime)^\ast\to(T_{Y^\prime})^\ast\to(\tr Z^\prime)^\ast\to0$ is also exact.
Therefore, the long exact sequence
$$
\cdots\to\Ext^{i-1}(\tr Z^\prime,R)\to\Ext^i(\tr X^\prime,R)\xrightarrow{\Ext^i(\tr f^\prime,R)}\Ext^i(\tr Y^\prime,R)\to\Ext^i(\tr Z^\prime,R)\to\cdots
$$
is obtained and it deduces the claim. 
\renewcommand{\qedsymbol}{$\square$}
\end{proof}

The implication $(5)\Rightarrow(1)$ follows from the above claim.

Now suppose that $f$ satisfies $(\t_1)$. 
Let $t:X\to P$ be a left $\proj R$-approximation of $X$.
Theorem \ref{T_1} implies that $\binom{f}{t}$ is injective.
Also, $\binom{f}{t}^\ast$ is surjective.
Hence the implication $(2)\Rightarrow(3)$ holds.
If $f$ satisfies $(\t_n)$, then it follows from the above claim that $\Cokv  f=\Cok \binom{f}{t}$ is $(n-1)$-torsionfree.
Thus the implication $(1)\Rightarrow(2)$ follows. 

The implications $(3)\Rightarrow(4), (3)\Rightarrow(5)$ and $(4)\Rightarrow(6)$ are clear. 

Assume that there is an exact sequence $0\to X^\prime\xrightarrow{f^\prime}Y^\prime\to Z^\prime\to0$ such that $f\approx f'$, $Z'$ is $(n-1)$-torsionfree, $f'^\ast$ is surjective and $\sgrade\Ext^1(Z',R)\ge n-1$.
We take a left $\proj R$-approximation $t:X^\prime\to P$ of $X$ and construct the pushout diagram:
$$
\xymatrix@R-1pc@C-1pc{
0\ar[rr]&&X^\prime\ar[rr]^{f^\prime}\ar[dd]^{t}&&Y^\prime\ar[rr]\ar[dd]&&Z^\prime\ar[rr]\ar@{=}[dd]&&0  \\
&&&&&&&& \\
0\ar[rr]&&P\ar[rr]&&W\ar[rr]&&Z^\prime\ar[rr]&&0.
}
$$
Then $W$ is $(n-1)$-torsionfree by \cite[Theorem 1.1]{AR}.
There is an exact sequence $0\to X^\prime\xrightarrow{\binom{f^\prime}{t}}Y^\prime\oplus P\to Z^\prime\to0$ with $\binom{f}{t}^\ast$ surjective.
Thus $f^\prime$ satisfies the condition $(3)$, and so $f$ satisfies the condition $(1)$.

Finally, we assume that the implication $(6)\Rightarrow(1)$ holds for every homomorphism from an $(n-1)$-torsionfree module, 
and $0\to X\xrightarrow{f}Y\to Z\to0$ is an exact sequence with $X$ and $Z$ in $\tf_{n-1}(R)$.
Since $f$ satisfies $(6)$, $f$ satisfies $(\t_n)$.
By Remark \ref{bydef}(4), $Y$ is $(n-1)$-torsionfree, and so $\tf_{n-1}(R)$ is closed under extensions.
\end{proof}

For a homomorphism $f:X\to Y$ of $R$-modules, we decompose $f$ as $X\overset{\tilde{f}}{\twoheadrightarrow}\image f\hookrightarrow Y$.
It immediately follows from Theorem \ref{T_1} that $f$ satisfies $(\t_1)$ if and only if so does $\tilde{f}$.
For the condition $(\t_n)$ where $n$ is arbitrary, the following holds. 

\begin{cor}\label{ftilde}
Let $f:X\to Y$ be a homomorphism of $R$-modules and $n\ge0$ be an integer. 
Decompose $f$ as $X\xrightarrow{\tilde{f}}\image{f}\xrightarrow{\iota}Y$.
Suppose that $\Cok f$ is $(n-1)$-torsionfree and $\sgrade_{R^{\mathrm{op}}}\Ext_R^1(\Cok f,R)\ge n-1$.
Then $f$ satisfies $(\t_n)$ if and only if $\tilde{f}$ satisfies $(\t_n)$. 
\end{cor}

\begin{proof}
Applying the implication $(6)\Rightarrow(1)$ in Theorem \ref{main1} to the exact sequence $0\to\image f\xrightarrow{\iota}Y\to\Cok f\to0$, we see that $\iota$ satisfies $(\t_n)$.
Therefore, the assertion follows from Remark \ref{bydef}(2)(3).
\end{proof}

The following remark is a collection of several observations about Corollary \ref{ftilde}.
It also gives some examples for the condition $(\t_n)$.

\begin{rmk}\label{ftildermk}
Let $f:X\to Y$ be a homomorphism of $R$-modules, and decompose $f$ as $X\xrightarrow{\tilde{f}}\image f\xrightarrow{\iota} Y$.
\begin{enumerate}[\rm(1)]
    \item
    The homomorphism $f$ satisfies $(\t_1)$ if and only if so does $\tilde{f}$.
    This is the situation where we let $n=1$ in Corollary \ref{ftilde}.
    \item 
    If $f$ satisfies $(\t_2)$, then so does $\tilde{f}$.
    Indeed, $\Ext^1(\tr\iota,R)$ is surjective as $\Ext^1(\tr f,R)$ is bijective.
    Since $\iota$ is injective, so is $\Ext^1(\tr\iota,R)$ by Theorem \ref{T_1}.
    Hence $\Ext^1(\tr\iota.R)$ is bijective, and so is $\Ext^1(\tr\tilde{f},R)$.
    Moreover, as $\Ext^2(\tr\tilde{f},R)$ is injective, $\tilde{f}$ satisfies $(\t_2)$.
    \item
    If $R$ has positive injective dimension, then there exists an $R$-homomorphism $f$ which does not satisfy $(\t_2)$ such that $\tilde{f}$ satisfies $(\t_\infty)$.
    In fact, then there is an $R$-module $X$ which is not torsionless.
    Let $f=\binom{1}{0}:R\to R\oplus X$.
    The homomorphism $\tilde{f}$ satisfies $(\t_\infty)$, but $f$ does not satisfy $(\t_2)$ by Remark \ref{bydef}(4).
    \item
    Let $n\ge1$ be an integer. 
    If $\tf_{n-1}(R)$ is not closed under extensions, then there exists an $R$-homomorphism $f$ such that $\Cok f$ is $(n-1)$-torsionfree, $\tilde{f}$ satisfies $(\t_\infty)$, and $f$ does not satisfy $(\t_n)$.
    Indeed, there is an exact sequence $0\to X\xrightarrow{f}Y\to Z\to0$, where $X$ and $Z$ are $(n-1)$-torsionfree but $Y$ is not so.
    Hence $f$ satisfies the required condition.
    \item
    Let $n\ge3$ be an integer.
    If $\tf_n(R)\subsetneq\tf_1(R)$ holds, then there exists an $R$-homomorphism $f$ which satisfies $(\t_\infty)$ such that $\tilde{f}$ does not satisfy $(\t_n)$.
    In fact, we take an $R$-module $X$ which is not $n$-torsionfree but torsionless.
    Then there is a monomorphism $f_1:X\to P$ with $P$ projective.
    Moreover, We choose an epimorphism $f_2:Q\to X$ with $Q$.
    The composite $f=f_1 f_2$ satisfies the desired condition.
\end{enumerate}
\end{rmk}

By Theorem \ref{main2}, if a homomorphism $f:X\to Y$ satisfies $(\t_1)$, then $\Kerv f$ is torsionless (and, therefore, so is $\Ker f$).
However, even if $f$ satisfies $(\t_2)$, $\Kerv f$ does not always become reflexive.

\begin{ex}
Suppose that there exists a second syzygy module $X=\syz^2 C$ such that $X\notin\tf_2(R)$.
There is an exact sequence $0\to X\to P\xrightarrow{f}\syz C\to0$ with $P$ projective.
Then $f$ satisfies $(\t_2)$, but $\Kerv f\approx\Ker f\cong X$ is not reflexive.
\end{ex}

From the following proposition, if a homomorphism $f:X\to Y$ of $R$-modules satisfies $(\t_n)$, then $\Kerv f$ becomes an $n$-syzygy module. 

\begin{cor}\label{syztrf}
Let $n\ge1$ be an integer and $f:X\to Y$ a homomorphism of $R$-modules.
If $f$ satisfies $(\t_n)$, then $\Kerv f$ is a first syzygy of an $(n-1)$-torsionfree module.
\end{cor}

\begin{proof}
By Theorem \ref{main1} and Lemma \ref{diag}, $\Cokv f$ is $(n-1)$-torsionfree and $\Kerv f\approx\syz\Cokv f$.
\end{proof}

We consider the case where $R$ is commutative.
If a homomorphism $f:X\to Y$ satisfies $(\t_n)$, the kernel and the cokernel of $f$ are related as follows.
It gives a partial generalization of Corollary \ref{KKmain}, which is none other than \cite[Theorem 4.12]{KK}.

\begin{prop}\label{T_nS_n}
Suppose that $R$ is commutative.
Let $n\ge1$ be an integer and $f:X\to Y$ a homomorphism of $R$-modules.
Suppose that $f$ satisfies $(\t_n)$.
\begin{enumerate}[\rm(1)]
    \item
    Assume that $R$ satisfies $(\s_n)$.
    If $\Cok f$ satisfies $(\s_{n-2})$, then $\Ker f$ satisfies $(\s_n)$.
    \item
    Assume that $R$ satisfies $(\s_n)$ and $(\g_{n-1})$.
    If $\Cok f$ satisfies $(n-2)$-torsionfree, then $\Ker f$ is $n$-torsionfree.
\end{enumerate}
\end{prop}

\begin{proof}
(1) It follows from Remark \ref{tfrmk} and Theorem \ref{main1} that $\Cokv f$ satisfies $(\s_{n-1})$.
Thus the assertion follows by Lemma \ref{diag} and the depth lemma.

(2) The assertion follows from (1) and \cite[Theorem 3.8]{EG}.
\end{proof}


\section{Grade inequalities of Ext modules}
In this section, as an application of Theorem \ref{main2}, we consider the grades of Ext modules.
Let $M$ be an $R$-module and $n\ge 1$ an integer.
Auslander and Bridger \cite{AB} state and prove a criterion for $\syz^i M$ to be $i$-torsionfree for $1\le i\le n$.
By using Theorem \ref{main2}, we can recover \cite[Proposition 2.26]{AB}, which is the most fundamental theorem in studies on $n$-torsionfree modules.

\begin{cor}[Auslander--Bridger]\label{ABProp2.26}
Let $n\ge1$ be an integer and $M$ an $R$-module.
The following are equivalent.
\begin{enumerate}[\rm(1)]
    \item 
    The inequality $\grade_{R^{\mathrm{op}}}\Ext_R^{i}(M,R)\ge i-1$ holds for all $1\le i\le n$.
    \item
    The syzygy $\syz^i M$ is $i$-torsionfree for all $1\le i\le n$.
\end{enumerate}
\end{cor}

\begin{proof}
We use induction on $n$.
The assertion is trivial for $n=1$.
Let $n>1$.
Assume that $(1)$ or $(2)$ holds.
Then $\syz^{i}M$ is $i$-torsionfree for all $1\le i\le n-1$ by the induction hypothesis.
Let $P\xrightarrow{f}\syz^{n-1}M$ be an epimorphism with $P$ projective.
By Remark \ref{bydef}(4), $f$ satisfies $(\t_n)$.
It follows from Theorem \ref{main2} that $\grade\Ker\Ext^1(f,R)\ge n-1$ if and only if $\Ker f$ is $n$-torsionfree.
As $\Ker\Ext^1(f,R)\cong\Ext^1(\syz^{n-1}M,R)\cong\Ext^n(M,R)$ and $\Ker f\approx\syz^n M$, we have the desired result. 
\end{proof}

Let $X, Y$ be $R$-modules and $n\ge0$ an integer, and take projective resolutions $\cdots\to F_1\to F_0\to X\to0$ and $\cdots\to G_1\to G_0\to Y\to0$.
Furthermore, we take a projective resolution of $\tr X$ of the following form:
$$
\cdots\to H_3\to H_2\to F_0^\ast\to F_1^\ast\to\tr X\to0.
$$
Let $\alpha:X\to\syz^n Y$ be an $R$-homomorphism.
There are $R$-homomorphisms $\alpha_0$ and $\alpha_1$ such that the left diagram below commutes, and we extend $\alpha_0^\ast$ and $\alpha_1^\ast$ to a chain map as in the right diagram below.

$$
\xymatrix@R-1pc@C-1pc{
F_1\ar[r]\ar[d]^{\alpha_1}&F_0\ar[r]\ar[d]^{\alpha_0}&X\ar[r]\ar[d]^{\alpha}&0&&&G_0^\ast\ar[r]\ar[d]^{\beta_{n+1}}&G_1^\ast\ar[r]\ar[d]^{\beta_n}&\cdots\ar[r]&G_{n-1}^\ast\ar[r]\ar[d]^{\beta_2}&G_n^\ast\ar[r]\ar[d]^{\alpha_0^\ast}&G_{n+1}^\ast\ar[d]^{\alpha_1^\ast} \\
G_{n+1}\ar[r]&G_n\ar[r]&\syz^n Y\ar[r]&0&&&H_{n+1}\ar[r]&H_n\ar[r]&\cdots\ar[r]&H_2\ar[r]&F_0^\ast\ar[r]&F_1^\ast
}
$$
Hence we get a commutative diagram with exact rows:
$$
\xymatrix@R-1pc@C-1pc{
H_n^\ast\ar[r]\ar[d]^{\beta_n^\ast}&H_{n+1}^\ast\ar[r]\ar[d]^{\beta_{n+1}^\ast}&\tr\syz^n\tr X\ar[r]\ar[d]^{\theta_{X,Y}^n(\alpha)}&0 \\
G_1\ar[r]&G_0\ar[r]&Y\ar[r]&0.
}
$$
Then the following lemma holds.

\begin{lem}\cite[Corollary 3.4]{AR}\label{ARCor3.4}
The above assignment $\underline{\alpha}\mapsto\underline{\theta_{X,Y}^n(\alpha)}$ gives a natural isomorphism
$$
\lhom_R(X,\syz^n Y)\xrightarrow{\sim}\lhom_R(\tr\syz^n\tr X,Y).
$$
\end{lem}

Let $M$ be an $R$-module and $n\ge1$ an integer.
We denote by $\J_n^2$ the functor $\tr\syz^n\tr\syz^n$.
Let $\cdots\to P_1\to P_0\to M\to0$ be a projective resolution, and take a projective resolution of $\tr\syz^n M$ of the form
$$
\cdots\to Q^n_{3}\to Q^n_{2}\to P_n^\ast\to P_{n+1}^\ast\to\tr\syz^n M\to0.
$$
There are chain maps $\{ f_i\}$, $\{ g_i\}$ given in the left diagram below, which induce the right commutative diagram.

\footnotesize
$$
\xymatrix@R-1pc@C-1pc{
P_0^\ast\ar[r]\ar[d]^{f_n}&P_1^\ast\ar[r]\ar[d]^{f_{n-1}}&P_2^\ast\ar[r]\ar[d]^{f_{n-2}}&\cdots\ar[r]&P_{n-2}^\ast\ar[r]\ar[d]^{f_2}&P_{n-1}^\ast\ar[r]\ar@{=}[d]&P_n^\ast\ar@{=}[d]&&&{Q_n^n}^\ast\ar[r]\ar[d]^{g_{n-1}^\ast}&{Q_{n+1}^n}^\ast\ar[r]\ar[d]^{g_n^\ast}&\J_n^2 M\ar[r]\ar[d]^{\psi_M^{n,n-1}}&0 \\
Q_n^{n-1}\ar[r]\ar[d]^{g_n}&Q_{n-1}^{n-1}\ar[r]\ar[d]^{g_{n-1}}&Q_{n-2}^{n-1}\ar[r]\ar[d]^{g_{n-2}}&\cdots\ar[r]&Q_2^{n-1}\ar[r]\ar[d]^{g_2}&P_{n-1}^\ast\ar[r]\ar[d]^{g_1}&P_n^\ast\ar[r]\ar@{=}[d]&P_{n+1}^\ast\ar@{=}[d]&&{Q_{n-1}^{n-1}}^\ast\ar[r]\ar[d]^{f_{n-1}^\ast}&{Q_n^{n-1}}^\ast\ar[r]\ar[d]^{f_n^\ast}&\J_{n-1}^2 M\ar[r]\ar[d]^{\psi_M^{n-1}}&0 \\
Q_{n+1}^n\ar[r]&Q_n^n\ar[r]&Q_{n-1}^n\ar[r]&\cdots\ar[r]&Q_3^n\ar[r]&Q_2^n\ar[r]&P_n^\ast\ar[r]&P_{n+1}^\ast&&P_1\ar[r]&P_0\ar[r]&M\ar[r]&0
}
$$
\normalsize
We obtain the maps $\psi^{n,n-1}_M:\J^2_n M\to \J^2_{n-1}M$ and $\psi^{n-1}_M:\J^2_{n-1}M\to M$.
Note that $\underline{\psi^{n,n-1}_M}$ and $\underline{\psi^{n-1}_M}$ in $\lmod R$ are uniquely determined by $M$.
We take a projective resolution of $\J^2_{n-1} M$ of the form
$$
\cdots\to S_3\to S_2\to {Q_{n-1}^{n-1}}^\ast\to {Q_n^{n-1}}^\ast\to \J^2_{n-1} M\to0.
$$
There is a chain map $\left\{h_i \right\}$ as in the left diagram below, which yields the right commutative diagram.
\footnotesize
$$
\xymatrix@R-1pc@C-1pc{
P_{n+1}\ar[r]\ar[d]^{h_{n+1}}&P_n\ar[r]\ar[d]^{h_n}&P_{n-1}\ar[r]\ar[d]^{h_{n-1}}&{Q_2^{n-1}}^\ast\ar[r]\ar[d]^{h_{n-2}}&\cdots\ar[r]&{Q_{n-2}^{n-1}}^\ast\ar[r]\ar[d]^{h_2}&{Q_{n-1}^{n-1}}^\ast\ar[r]\ar@{=}[d]&{Q_n^{n-1}}^\ast\ar@{=}[d]&&P_{n+1}\ar[r]\ar[d]^{h_{n+1}}&P_n\ar[r]\ar[d]^{h_n}&\syz^n M\ar[r]\ar[d]^{\gamma_M}&0\\
S_{n+1}\ar[r]&S_n\ar[r]&S_{n-1}\ar[r]&S_{n-2}\ar[r]&\cdots\ar[r]&S_2\ar[r]&{Q_{n-1}^{n-1}}^\ast\ar[r]&{Q_n^{n-1}}^\ast&&S_{n+1}\ar[r]&S_n\ar[r]&\syz^n \J^2_{n-1} M\ar[r]&0
}
$$
\normalsize
We have that $\underline{\psi^{n,n-1}_{M}}=\underline{\theta^n_{\syz^n M,\J^2_{n-1}M}(\gamma_M)}$ in $\lmod R$.
Indeed, the commutative diagram in the lower left gives the diagram in the lower right.
\footnotesize
$$
\xymatrix@R-1pc@C-1pc{
Q_n^{n-1}\ar[r]\ar@{=}[d]&Q_{n-1}^{n-1}\ar[r]\ar@{=}[d]&S_2^\ast\ar[r]\ar[d]^{h_2^\ast}&\cdots\ar[r]&S_{n-2}^\ast\ar[r]\ar[d]^{h_{n-2}^\ast}&S_{n-1}^\ast\ar[r]\ar[d]^{h_{n-1}^\ast}&S_n^\ast\ar[r]\ar[d]^{h_n^\ast}&S_{n+1}^\ast\ar[d]^{h_{n+1}^\ast}&&{Q_n^n}^\ast\ar[r]\ar[d]^{g_{n-1}^\ast}&{Q_{n+1}^n}^\ast\ar[r]\ar[d]^{g_n^\ast}&\J_n^2 M\ar[r]\ar[d]^{\psi_M^{n,n-1}}&0 \\
Q_n^{n-1}\ar[r]\ar[d]^{g_n}&Q_{n-1}^{n-1}\ar[r]\ar[d]^{g_{n-1}}&Q_{n-2}^{n-1}\ar[r]\ar[d]^{g_{n-2}}&\cdots\ar[r]&Q_2^{n-1}\ar[r]\ar[d]^{g_2}&P_{n-1}^\ast\ar[r]\ar[d]^{g_1}&P_n^\ast\ar[r]\ar@{=}[d]&P_{n+1}^\ast\ar@{=}[d]&&{Q_{n-1}^{n-1}}^\ast\ar[r]\ar@{=}[d]&{Q_n^{n-1}}^\ast\ar[r]\ar@{=}[d]&\J_{n-1}^2 M\ar[r]\ar@{=}[d]&0 \\
Q_{n+1}^{n}\ar[r]&Q_n^n\ar[r]&Q_{n-1}^n\ar[r]&\cdots\ar[r]&Q_3^n\ar[r]&Q_2^n\ar[r]&P_n^\ast\ar[r]&P_{n+1}^\ast&&{Q_{n-1}^{n-1}}^\ast\ar[r]&{Q_n^{n-1}}^\ast\ar[r]&\J_{n-1}^2 M\ar[r]&0
}
$$
\normalsize
By Lemma \ref{ARCor3.4} (and the observation stated before it), we have that $\underline{\theta^n_{\syz^n M,J}(\gamma_M)}=\underline{1_J}\circ \underline{\psi^{n,n-1}_{M}}=\underline{\psi^{n,n-1}_{M}}$ in $\lmod R$, where $J=\J^2_{n-1}M$.

We define a map $\Phi_M:\lhom_R(\syz^n M,\syz^n M)\to\lhom_R(\J^2_n M,\J^2_{n-1}M)$ by the composition $\lhom_R(\syz^n M,\syz^n M)\xrightarrow{\lhom(\syz^n M,\gamma_M)}\lhom_R(\syz^n M,\syz^n \J^2_{n-1}M)\xrightarrow{\theta^n_{\syz^n M,\J^2_{n-1}M}}\lhom_R(\J^2_n M,\J^2_{n-1}M)$.
We get the following lemma.

\begin{lem}\label{unit}
With the notation above, the following hold.
\begin{enumerate}[\rm(1)]
    \item
    One has $\Phi_M(\underline{1_{\syz^n M}})=\underline{\psi^{n,n-1}_M}$ in $\lmod R$.
    \item
    One has $\underline{\psi_M^n}=\underline{\psi_M^{n-1}}\underline{\psi_M^{n,n-1}}$ in $\lmod R$.
    \item
    One has $\psi^{n,n-1}_M\approx\tr\syz\tr\psi^{n-1,n-2}_{\syz M}$.
\end{enumerate}
\end{lem}

From Theorem \ref{main2} and a basic property of the map $\psi^1_M$, we can deduce \cite[Theorem 1.1]{AR} if $R$ is commutative.
It is a well-known result about the extension closedness of the category of $n$-torsionfree modules.

\begin{cor}[Auslander--Reiten]\label{ARThm1.1}
Suppose that $R$ is commutative.
Let $n\ge1$ be an integer and $Z$ an $n$-torsionfree module.
Then the following are equivalent.
\begin{enumerate}[\rm(1)]
    \item
    The inequality $\grade_R\Ext_R^1(Z,R)\ge n$ holds.
    \item
    If $0\to X\to Y\to Z\to0$ is exact with $X$ $n$-torsionfree, then $Y$ is $n$-torsionfree.
    \item
    If $0\to P\to W\to Z\to0$ is exact with $P$ projective, then $W$ is $n$-torsionfree.
\end{enumerate}
\end{cor}

\begin{proof}
Assume that the condition $(1)$ holds, and let $0\to X\to Y\xrightarrow{f}Z\to0$ be an exact sequence with $X$ $n$-torsionfree.
As $\grade\Ker\Ext^1(f,R)\ge\grade\Ext^1(Z,R)\ge n$ and $\Ker f\cong X$ is $n$-torsionfree, $f$ satisfies $(\t_n)$ by Theorem \ref{main2}.
It follows from Remark \ref{bydef}(4) that $Y$ is $n$-torsionfree.
Thus the implication $(1)\Rightarrow(2)$ is shown.

The implication $(2)\Rightarrow(3)$ is clearly true.
Suppose that the condition $(3)$ is satisfied.
By \cite[Proposition 2.21]{AB}, there is an exact sequence $0\to P\to \J_1^2 Z\oplus Q\xrightarrow{(\psi_Z^1, s)}Z\to0$, where $P$ and $Q$ are projective.
Let $f=(\psi_Z^1, s)$.
We have that $\J_1^2 Z$ is $n$-torsionfree by assumption.
In particular, $f$ satisfies $(\t_n)$.
Moreover, since $\Ker f\cong P$ is projective, $\grade\Ker\Ext^1(f,R)\ge n$ by Theorem \ref{main2}.
As $\Ext^1(\J_1^2 Z, R)=0$, we have that $\grade\Ext^1(Z,R)=\grade\Ker\Ext^1(f,R)\ge n$, that is, $(1)$ follows.
\end{proof}

We show some lemmas in order to prove the main result of this section.

\begin{lem}\label{psicom}
Let $M$ be an $R$-module and $m, n\ge0$ be integers.
Then $\psi_{M}^{i,i-1}$ satisfies $(\t_n)$ for all $1\le i\le m$ if and only if so does $\psi_{M}^i$ for all $1\le i\le m$.
\end{lem}

\begin{proof}
The assertion immediately follows from Remark \ref{bydef}(2)(3) and Lemma \ref{unit}(2).
\end{proof}

\begin{lem}\label{extex}
Let $n\ge2$ be an integer and $M$ an $R$-module. 
If $\Ext_{R^{\mathrm{op}}}^{n-1}(\Ext_R^n(M,R),R)=0$ $($e.g. $\grade_{R^{\mathrm{op}}}\Ext_R^n(M,R)\ge n$$)$, then
there exists an exact sequence $0\to E\to J\xrightarrow{\psi} J^\prime\to0$ such that $E\approx\tr\syz^{n-2}\Ext_R^n(M,R)$, $J\approx  \J_n^2 M$, $J^\prime\approx \J_{n-1}^2 M$ and $\psi\approx\psi_M^{n,n-1}$.
\end{lem}

\begin{proof}
 By \cite[Proposition 2.21]{AB}, there is an exact sequence $0\to P\to \J_1^2\syz^{n-1}M\oplus Q\xrightarrow{(\psi_{\syz^{n-1}M}^1, s)}\syz^{n-1}M\to0$, where $P$ and $Q$ projective.
Hence \cite[Lemma 3.9]{AB} gives a commutative diagram with exact rows
\small
$$
\xymatrix@R-1pc@C-1pc{
(\J_1^2 \syz^{n-1}M\oplus Q)^\ast\ar[rr]\ar@{=}[dd]&&P^\ast\ar[rr]\ar@{=}[dd]&&\Ext^1(\syz^{n-1}M,R)\ar[rr]&&\Ext^1(\J_1^2 \syz^{n-1}M,R)&&&  \\
&&&&&&&&&& \\
(\J_1^2 \syz^{n-1}M\oplus Q)^\ast\ar[rr]&&P^\ast\ar[rr]&&\tr\syz^{n-1}M\ar[rr]^-{t}&&T\ar[rr]&&\tr P\ar[rr]&&0,
}
$$
\normalsize
where $T\approx\tr \J^2_1\syz^{n-1}M\approx\syz\tr\syz^n M$ and $t\approx\tr\psi^1_{\syz^{n-1}M}$
, and so we obtain that $0\to\Ext^n(M,R)\to\tr\syz^{n-1}M\xrightarrow{t}T\to0$ is exact as $\tr P\approx0$ and $\Ext^1(\J_1^2 \syz^{n-1}M,R)=0$. Applying Remark \ref{rmkst}(3) to this repeatedly, we get an exact sequence
\begin{equation}\label{extex1}
\xymatrix@R-1pc@C-1pc{
0\ar[rr]&&\syz^{n-1}\tr\syz^{n-1}M\oplus Q_1\ar[rr]^-{u}&&U\ar[rr]&&\syz^{n-2}\Ext^n(M,R)\oplus Q_2\ar[rr]&&0,
}
\end{equation}
where $Q_1, Q_2$ are projective, $U\approx\syz^{n-1}T\approx\syz^n\tr\syz^n M$ and $u\approx\syz^{n-1}t\approx\syz^{n-1}\tr\psi^1_{\syz^{n-1}M}$.
Since $\Ext^1(\syz^{n-2}\Ext^n(M,R),R)\cong\Ext^{n-1}(\Ext^n(M,R),R)=0$, the $R$-dual map $u^\ast$ is surjective.
Applying \cite[Lemma 3.9]{AB} to (\ref{extex1}) gives an exact sequence $0\to E\to J\xrightarrow{\psi}J^\prime\to0$, where $E\approx\tr\syz^{n-2}\Ext_R^n(M,R)$, $J\approx \J_n^2 M$, $J^\prime\approx \J_{n-1}^2 M$ and $\psi\approx\tr u\approx\tr\syz^{n-1}\tr\psi_{\syz^{n-1}M}^1$.
As $\psi_M^{n,n-1}\approx\tr\syz^{n-1}\tr\psi_{\syz^{n-1}M}^1$ by Lemma \ref{unit}(3), we are done.
\end{proof}

Let $m,n$ be nonnegative integers.
By $\G_{m,n}$ we denote the full subcategory of $\mod R$ consisting of modules $M$ such that $\Ext^i_R(M,R)=0$ for all $1\le i\le m$ and $\Ext^j_{R^{\mathrm{op}}}(\tr M,R)=0$ for all $1\le j\le n$.

The results \cite[Proposition 2.26 and Corollary 2.32]{AB} describe the relationship between the grades of $\Ext$ modules and the torsionfreeness of syzygy modules, and the relationship of them with $\psi_M^i$ being represented by monomorphisms.
These results are the basis of the theories of spherical modules and Gorenstein dimension developed in \cite{AB}.
The main result of this section is the following theorem, which can be regarded as a higher version of those results.

\begin{thm}\label{gradethm}
Let $n\ge1$ and $j\ge0$ be integers and $M$ an $R$-module. 
The following are equivalent.
\begin{enumerate}[\rm(1)]
     \item
     The inequality $\grade_{R^{\mathrm{op}}}\Ext_R^i(M,R)\ge i+j-1$ holds for all $1\le i\le n$.
     \item
     The syzygy $\syz^i M$ is $i$-torsionfree and the map $\psi_M^i:\J_i^2 M\to M$ satisfies $(\t_j)$ for all $1\le i\le n$.
     \item
     The syzygy $\syz^i M$ is $i$-torsionfree and the map $\psi_M^{i,i-1}:\J_i^2 M\to \J_{i-1}^2 M$ satisfies $(\t_j)$ for all $1\le i\le n$.
\end{enumerate}
\end{thm}

\begin{proof}
The assertion of the theorem is the same as Corollary \ref{ABProp2.26} if $j=0$.
Thus we may assume that $j>0$.
Lemma \ref{psicom} deduces the equivalence $(2)\Leftrightarrow(3)$. 
We prove the remaining part of the theorem step by step.

(i) First we deal with the case $n=1$. 
By \cite[Proposition 2.21]{AB}, there exists an exact sequence $0\to P\to \J_1^2 M\oplus Q\xrightarrow{f}M\to0$ with $P, Q$ projective such that $f\approx \psi_M^1$.
Therefore, $\psi_M^1$ satisfies $(\t_j)$ if and only if $\grade\Ker\Ext^1(\psi_M^1,R)\ge j$ by Theorem \ref{main2}.
As $\grade\Ker\Ext^1(\psi_M^1,R)=\grade\Ext^1(M,R)$ by $\Ext^1(\J_1^2 M,R)=0$, the assertion follows.

(ii) We prove the implication $(1)\Rightarrow(3)$.
Fix an integer $1\le i\le n$.
Since $\syz^i M$ is $i$-torsionfree by Corollary \ref{ABProp2.26}, it suffices to prove that $\psi_M^{i,i-1}$ satisfies $(\t_j)$.
As the assertion holds for $i=1$ by (i), we may assume that $i>1$.
Since $\grade\Ext^i(M,R)\ge i+j-1\ge i\ge2$, $\Kerv \psi_M^{i,i-1}\approx\tr\syz^{i-2}\Ext^i(M,R)$ by Lemma \ref{extex} and Remark \ref{rmkst}.
Now $\Ext^i(M,R)\in\G_{i+j-2,0}$ by the assumption.
Thus $\tr\syz^{i-2}\Ext^i(M,R)\in\G_{i-2,j}$ by \cite[Proposition 1.1.1]{I}, that is, $\Kerv\psi_M^{i,i-1}$ is $j$-torsionfree.
Moreover, $\grade\Ker\Ext^1(\psi_M^{i,i-1},R)=\infty$ as $\Ext^1(\J_{i-1}^2 M,R)=0$.
It follows from Theorem \ref{main2} that $\psi_M^{i,i-1}$ satisfies $(\t_j)$.

(iii) Finally, we show that the implication $(2)\Rightarrow(1)$ holds.
Recall that $(2)$ is equivalent to $(3)$.

(iii-a) Suppose $j=1$. 
By \cite[Proposition 2.21]{AB}, there exists an exact sequence $0\to B\to \J_i^2 M\oplus P\xrightarrow{(\psi_M^i,s)}M\to0$ such that $P$ is projective and $B$ has projective dimension at most $i-1$.
Hence there exists an exact sequence $0\to\syz^{i-1}B\to J\xrightarrow{f}\syz^{i-1}M\to0$ such that $J\approx\syz^{i-1}\J^2_{i}M$ and $f\approx\syz^{i-1}\psi^i_{M}$.
Thus $\Kerv\syz^{i-1}\psi_M^i$ is projective.
We establish a claim.

\begin{claim}
The map $\syz^{i-1}\psi_M^i$ satisfies $(\t_i)$.
\end{claim}

\begin{proof}[Proof of Claim]
The assumption and \cite[Proposition 1.1.1]{I} deduce that $\syz^{i-1}M$ and $\syz^{i-1}\J_i^2 M$ are $(i-1)$-torsionfree.
In particular, $\Ext^k(\tr\syz^{i-1}\psi_M^i,R)$ is bijective for all $1\le k\le i-1$.
We consider the following commutative diagram in $\lmod R$, which is obtained from the naturality of $\psi_{(-)}^{i-1}$:
$$
\xymatrix@R-1pc@C-1pc{
\J_{i-1}^2 \J_i^2 M\ar[rrr]^-{\underline{\psi_{\J_i^2 M}^{i-1}}}\ar[dd]^{\underline{\J_{i-1}^2\psi_M^i}}&&&\J_{i}^2 M\ar[dd]^{\underline{\psi_M^{i}}} \\
&&& \\
\J_{i-1}^2 M\ar[rrr]^-{\underline{\psi_M^{i-1}}}&&&M.
}
$$
By \cite[Proposition 1.1.1]{I}, we have $\J_i^2 M\in\G_{i,0}$, and thus $\psi_{\J_i^2 M}^{i-1}$ is a stable isomorphism.
Since $\psi_M^i$ satisfies $(\t_1)$, it follows from Remark \ref{bydef}(2) that $\J_{i-1}^2\psi_M^i$ satisfies $(\t_1)$.
Namely $\Ext^1(\tr \J^2_{i-1}\psi^i_M, R)$ is injective.
As $\tr \J^2_{i-1}\psi^i_M\approx \syz^{i-1}\tr\syz^{i-1}\psi^i_M$, $\Ext^i(\tr\syz^{i-1}\psi_M^i,R)$ is injective,
and so $\syz^{i-1}\psi_M^i$ satisfies $(\t_i)$.
\renewcommand{\qedsymbol}{$\square$}
\end{proof}

According to the above claim and Theorem \ref{main2}, we get
$$
i\le\grade\Ker(\Ext^1(\syz^{i-1}\psi_M^i,R):\Ext^1(\syz^{i-1}M,R)\to\Ext^1(\syz^{i-1}\J_i^2 M,R))=\grade\Ext^i(M,R),
$$
where the second equality is obtained as $\J_i^2 M\in\G_{i,0}$.
We get the conclusion.

(iii-b) We prove by induction on $j$ that $\grade\Ext^i(M,R)\ge i+j-1$ for all $1\le i\le n$.
We know that this holds true for $j=1$. 
Let $j>1$ and fix an integer $1\le i\le n$.
If $i=1$, then the claim follows from (i).
Thus we may assume that $i>1$.
Then $\grade\Ext^i(M,R)\ge i+(j-1)-1\ge i\ge2$ by the induction hypothesis, and so $\Kerv\psi_M^{i,i-1}\approx\tr\syz^{i-2}\Ext^i(M,R)$ by Lemma \ref{extex}.
Now $\psi_M^{i,i-1}$ satisfies $(\t_j)$ and $\grade\Ker\Ext^1(\psi_M^{i,i-1},R)=\infty$ as $\Ext^1(\J_{i-1}^2 M,R)=0$.
Because of Theorem \ref{main2}, $\tr\syz^{i-2}\Ext^i(M,R)$ is $j$-torsionfree.
Hence $0=\Ext^j(\tr\tr\syz^{i-2}\Ext^i(M,R),R)\cong\Ext^{i+j-2}(\Ext^i(M,R),R)$, and we conclude that $\grade\Ext^i(M,R)\ge i+j-1$.
\end{proof}

The following corollary deduces \cite[Proposition 2.28]{AB}.

\begin{cor}\label{gradethmcor}
Let $n\ge1$ and $j\ge1$ be integers and $M$ an $R$-module.
Suppose that $\grade_{R^{\mathrm{op}}}\Ext_R^i(M,R)\ge i+j-1$ for all $1\le i\le n$.
Then there are isomorphisms
$$
\Ext_{R^{\mathrm{op}}}^{n+k}(\tr\syz^n M,R)\cong\Ext_{R^{\mathrm{op}}}^{n-1+k}(\tr\syz^{n-1} M,R)\cong\cdots\cong\Ext_{R^{\mathrm{op}}}^{1+k}(\tr\syz M,R)\cong\Ext_{R^{\mathrm{op}}}^k(\tr M,R)
$$
for all $1\le k\le j-1$, there are monomorphisms
$$
\Ext_{R^{\mathrm{op}}}^{n+j}(\tr\syz^n M,R)\hookrightarrow\Ext_{R^{\mathrm{op}}}^{n-1+j}(\tr\syz^{n-1} M,R)\hookrightarrow\cdots\hookrightarrow\Ext_{R^{\mathrm{op}}}^{1+j}(\tr\syz M,R)\hookrightarrow\Ext_{R^{\mathrm{op}}}^j(\tr M,R)
$$
and the following implications hold.
\footnotesize
\begin{equation*}
\xymatrix@R-1pc@C-1pc{
M\in\tf_j(R)\ar@{=>}[d]&\ar@{=>}[r]&&\scalebox{0.7}[1]{$\syz M$}\in\tf_{j+1}(R)\ar@{=>}[d]&\ar@{=>}[r]&\cdots\cdots\ar@{=>}[r]&&\scalebox{0.7}[1]{$\syz^{n-1}M$}\in\tf_{n-1+j}(R)\ar@{=>}[d]&\ar@{=>}[r]&&\scalebox{0.7}[1]{$\syz^n M$}\in\tf_{n+j}(R)\ar@{=>}[d] \\
M\in\tf_{j-1}(R)\ar@{=>}[d]&\ar@{<=>}[r]&&\scalebox{0.7}[1]{$\syz M$}\in\tf_j(R)\ar@{=>}[d]&\ar@{<=>}[r]&\cdots\cdots\ar@{<=>}[r]&&\scalebox{0.7}[1]{$\syz^{n-1}M$}\in\tf_{n-2+j}(R)\ar@{=>}[d]&\ar@{<=>}[r]&&\scalebox{0.7}[1]{$\syz^n M$}\in\tf_{n-1+j}(R)\ar@{=>}[d] \\
\scalebox{1}[0.5]{$\vdots$}\ar@{=>}[d]&&&\scalebox{1}[0.5]{$\vdots$}\ar@{=>}[d]&&\cdots\cdots\cdots&&\scalebox{1}[0.5]{$\vdots$}\ar@{=>}[d]&&&\scalebox{1}[0.5]{$\vdots$}\ar@{=>}[d] \\
M\in\tf_1(R)&\ar@{<=>}[r]&&\scalebox{0.7}[1]{$\syz M$}\in\tf_2(R)&\ar@{<=>}[r]&\cdots\cdots\ar@{<=>}[r]&&\scalebox{0.7}[1]{$\syz^{n-1}M$}\in\tf_n(R)&\ar@{<=>}[r]&&\scalebox{0.7}[1]{$\syz^n M$}\in\tf_{n+1}(R)
}
\end{equation*}
\normalsize
\end{cor}

\begin{proof}
Let $1\le i\le n$ be an integer.
Note that $\tr \J^2_i M\approx\syz^i\tr\syz^i M$.
As $\psi_M^{i,i-1}$ satisfies $(\t_j)$, it follows from Theorem \ref{gradethm} that $\Ext^k(\tr\psi_M^{i,i-1},R):\Ext^k(\tr \J_i^2 M,R)\to \Ext^k(\tr\J_{i-1}^2 M,R)$ is an isomorphism for all $1\le k\le j-1$ and $\Ext^j(\tr\psi_M^{i,i-1},R):\Ext^j(\tr\J_i^2 M,R)\to \Ext^j(\tr\J_{i-1}^2 M,R)$ is injective.
Hence we have the desired isomorphisms and monomorphisms.
Since $\syz^i M$ is $i$-torsionfree for all $1\le i\le n$, we also get the desired implications.
\end{proof}

The following corollary provides a refinement of the implication $(1)\Rightarrow(2)$ in Corollary \ref{ABProp2.26}, which is the same as \cite[Proposition 2.26]{AB}.
In fact, it is obtained by letting $p=2$ in the following corollary.

\begin{cor}\label{gradepq}
Let $M$ be an $R$-module and $p,g$ integers such that $2\le p\le q$.
If $\grade_{R^{\mathrm{op}}}\Ext_R^i(M,R)\ge i-1$ for all $p\le i\le q$, then the following implications hold.
$$
\syz^{p-1}M:(p-1)\text{-torsionfree}\Longrightarrow\syz^p M:p\text{-torsionfree}\Longrightarrow\cdots\Longrightarrow\syz^q M:q\text{-torsionfree}.
$$
\end{cor}

\begin{proof}
Let $j=p-1$, $n=q-p+1$, and $N=\syz^j M$.
Then $n, j\ge1$ and the assumption means that $\grade\Ext^k(N,R)\ge k+j-1$ for all $1\le k\le n$.
Applying Corollary \ref{gradethmcor} to $N$ implies the assertion.
\end{proof}

\section{Gorenstein properties}

In this section, we provide a characterization of commutative Gorenstein rings by using results given in the previous section. Throughout this section, we assume that $R$ is commutative.

First, we consider the case where the Ext modules have finite length. 
The following proposition holds.

\begin{prop}\label{extlength}
Suppose that $R$ is local and with depth $t$.
Let $M$ be an $R$-module.
Suppose that $\Ext_R^i(M,R)$ has finite length for all $1\le i\le t$.
Then the following statements hold.
\begin{enumerate}[\rm(1)]
    \item
    The maps $\psi_M^{i,i-1}:\J_i^2 M\to \J_{i-1}^2 M$ and $\psi_M^i:\J_i^2 M\to M$ satisfy $(\t_1)$ for all $1\le i\le t$.
    \item
    Suppose that $\Ext_R^{t+1}(M,R)$ also has finite length.
    Then $\Ext_R^{t+1}(M,R)=0$ if and only if $\psi_M^{t+1,t}:\J_{t+1}^2 M\to \J_t^2 M$ satisfies $(\t_1)$, if and only if $\psi_M^{t+1}:\J^2_{t+1}M\to M$ satisfies $(\t_1)$.
\end{enumerate}
\end{prop}

\begin{proof}
(1) By \cite[Proposition 1.2.10]{BH}, we have $\grade\Ext^i(M,R)\ge t$ for all $1\le i\le t$.
The assertion follows from Theorem \ref{gradethm}.

(2) The assumption deduces that $\grade\Ext^i(M,R)\ge t$ for all $1\le i\le t+1$.
In particular, $\syz^i M$ is $i$-torsionfree for all $1\le i\le t+1$.
We have $\Ext^{t+1}(M,R)=0$ if and only if $\grade\Ext^i(M,R)\ge i$ for all $1\le i\le t+1$.
The assertion follows from (1) and Theorem $\ref{gradethm}$.
\end{proof}

Proposition \ref{extlength} can be strengthened in the case where $M$ is the residue field of the local ring $R$, as in the proposition below.
It also provides a criterion for $R$ to be Gorenstein in terms of morphisms represented by monomorphisms.

\begin{prop}\label{psik}
Suppose that $R$ is local and with depth $t$.
Let $k$ be the residue field of $R$.
\begin{enumerate}[\rm(1)]
    \item
    The maps $\psi_k^{i,i-1}:\J_i^2 k\to \J_{i-1}^2 k$ and $\psi_k^i:\J_i^2 k\to k$ are stable isomorphisms for all $1\le i\le t-1$.
    \item
    If $t>0$, then $\Ext_R^{i+1}(\tr\syz^i k,R)\cong k$ for all $0\le i\le t-1$ and $\Ext_R^{t+1}(\tr\syz^t k,R)=0$.
    In particular, $\psi_k^{t,t-1}:\J_{t}^2 k\to \J_{t-1}^2 k$ and $\psi_k^t:\J^2_t k\to k$ satisfy $(\t_1)$, but do not satisfy $(\t_2)$.
    \item
    The following are equivalent.
    \begin{enumerate}[\rm(i)]
        \item
        The ring $R$ is Gorenstein.
        \item
        The map $\psi^i_M:\J^2_i M\to M$ satisfies $(\t_1)$ for all $R$-modules $M$ and for all $i\ge1$.
        \item
        The map $\psi^{t+1}_M:\J^2_{t+1} M\to M$ satisfies $(\t_1)$ for all $R$-modules $M$.
        \item
        The map $\psi_k^{t+1}:\J^2_{t+1} k\to k$ satisfies $(\t_1)$.
        \item
        The map $\psi_k^{t+1,t}:\J^2_{t+1} k\to \J^2_t k$ satisfies $(\t_1)$.
    \end{enumerate}
\end{enumerate}
\end{prop}

\begin{proof}
(1) Since $\Ext^i(k,R)=0$ for all $i<t$, the assertion follows from \cite[Proposition 1.1.1]{I}.

(2) Let $0\le i\le t-1$ be an integer and take a free resolution $\cdots\to F_1\to F_0\to k\to0$ of $k$.
Dualizing it by $R$ gives rise to a free resolution
$$
0\to F_0^\ast\to F_1^\ast\to\cdots\to F_{i-1}^\ast\to F_i^\ast\to F_{i+1}^\ast\to\tr\syz^i k\to0
$$
of $\tr\syz^i k$.
Thus, $\Ext^{i+1}(\tr\syz^i k,R)\cong k$.
It follows from \cite[Theorem 4.1(2)]{DT} that $\Ext^{t+1}(\tr\syz^t k,R)=0$.
Since $\Ext^1(\tr \J^2_{t-1} k,R)\cong k\cong\Ext^1(\tr k,R)$ and $\Ext^1(\tr \J^2_t k,R)\cong0$, the maps $\psi^{t,t-1}_k$ and $\psi^t_k$ satisfy $(\t_1)$ but do not satisfy $(\t_2)$.

(3) The implications $(\text{ii})\Rightarrow(\text{iii})\Rightarrow(\text{iv})$ clearly hold.
The implication $(\text{iv})\Rightarrow(\text{v})$ follows from (1), (2) and Lemma \ref{psicom}.
It follows from \cite[Proposition 4.21 and Corollary 2.32]{AB} that the implication $(\text{i})\Rightarrow(\text{ii})$ holds.
By $(2)$, the condition $(\text{v})$ implies $\Ext^{t+2}(\tr\syz^{t+1}k,R)=0$.
The implication $(\text{v})\Rightarrow(\text{i})$ follows from \cite[Theorem 4.1(1)(3)]{DT}.
\end{proof}

For an integer $n\ge1$, we consider the situation where $R$ satisfies both $(\s_n)$ and $(\g_{n-1})$.
Various studies on this condition have been done so far; see \cite{AB, DT, EG, GT, MTT} for instance.
Below is a direct corollary of the above proposition, which gives a characterization of those rings.

\begin{cor}\label{S_nG_n-1}
Let $n\ge1$ be an integer.
The following are equivalent.
\begin{enumerate}[\rm(1)]
    \item
    The ring $R$ satisfies $(\s_n)$ and $(\g_{n-1})$.
    \item
    The map $\psi^{i}_M:\J^2_i M\to M$ satisfies $(\t_1)$ for all $1\le i\le n$.
    \item
    Let $\p$ be a prime ideal of $R$.
    If $\depth R_{\p}<n$, then $\psi^{i}_{R/{\p}}:\J^2_i ({R/{\p}})\to R/{\p}$ satisfies $(\t_1)$ for all $1\le i\le n$.
\end{enumerate}
\end{cor}

\begin{proof}
The implication $(2)\Rightarrow(3)$ clearly holds, and $(1)\Rightarrow(2)$ follows from \cite[Proposition 4.21 and Corollary 2.32]{AB}.
Assume that the condition $(3)$ holds.
We want to prove that $R$ satisfies $(\s_n)$ and $(\g_{n-1})$.
For this, it is enough to show that $R_{\p}$ is Gorenstein for each prime ideal $\p$ of $R$ with $\depth R_{\p}\le n-1$; see \cite[Theorem 2.3]{MTT} for instance.
Let $t=\depth R_{\p}$.
As $1\le t+1\le n$, the map $\psi^{t+1}_{R/{\p}}$ satisfies $(\t_1)$, and we have a stable isomorphism $\psi^{t+1}_{R/{\p}}\otimes R_{\p}\approx\psi^{t+1}_{R_{\p}/{\p R_{\p}}}$.
Hence $\psi^{t+1}_{R_{\p}/{\p R_{\p}}}$ satisfies $(\t_1)$.
By Proposition \ref{psik}, $R_{\p}$ is Gorenstein and we are done.
\end{proof}


\begin{ac}
The author would like to thank his supervisor Ryo Takahashi for a lot of valuable discussions and advice.

\end{ac}

\end{document}